\documentclass[11pt]{amsart}
\usepackage[utf8]{inputenc}
\usepackage[english]{babel}
\usepackage{amsmath}
\usepackage{amssymb}
\usepackage{amsthm}
\usepackage{latexsym}
\usepackage{mathtools}
\usepackage{amsfonts}
\usepackage{mathrsfs}
\usepackage{textcomp}
\usepackage[T1]{fontenc}
\usepackage{graphicx}
\usepackage{setspace}
\usepackage{nicefrac}
\usepackage{indentfirst}
\usepackage{enumerate}
\usepackage{wasysym}
\usepackage{upgreek}
\usepackage{hyperref}
\usepackage{booktabs}
\usepackage{etoolbox}
\usepackage{mathdots}
\usepackage{wrapfig}
\usepackage{floatflt}
\usepackage{tensor} 
\usepackage{parskip}
\usepackage{lscape}
\usepackage{stmaryrd}
\usepackage[enableskew]{youngtab}
\usepackage{ytableau}

\usepackage{tikz}
\usetikzlibrary{arrows,matrix}
\usetikzlibrary{positioning}
\usetikzlibrary{decorations}

\usepackage[all,cmtip]{xy}
\usepackage[labelfont=bf, font={small}]{caption}[2005/07/16]

\usepackage{xcolor}
\definecolor{red}{rgb}{1,0,0}

\newcommand{\xto}[1]{\xrightarrow{\phantom{a}{#1}{\phantom{a}}}}

\newcommand{\vvirg}{ , \dots , }
\newcommand{\ootimes}{ \otimes \cdots \otimes }
\newcommand{\ooplus}{ \oplus \cdots \oplus }
\newcommand{\ttimes}{ \times \cdots \times }

\newcommand{\textsum}{{\textstyle \sum}}
\newcommand{\textprod}{{\textstyle \prod}}

\newcommand{\bfA}{\mathbf{A}}

\newcommand{\bfF}{\mathbf{F}}

\newcommand{\bfP}{\mathbf{P}}
\newcommand{\bfQ}{\mathbf{Q}}
\newcommand{\bfR}{\mathbf{R}}

\newcommand{\bfV}{\mathbf{V}}

\newcommand{\bfa}{\mathbf{a}}

\newcommand{\bfd}{\mathbf{d}}

\newcommand{\bfk}{\mathbf{k}}

\newcommand{\calC}{\mathcal{C}}

\newcommand{\calF}{\mathcal{F}}

\newcommand{\calI}{\mathcal{I}}

\newcommand{\calQ}{\mathcal{Q}}

\newcommand{\calS}{\mathcal{S}}

\newcommand{\calW}{\mathcal{W}}

\newcommand{\calZ}{\mathcal{Z}}

\newcommand{\scrO}{\mathscr{O}}

\newcommand{\bbC}{\mathbb{C}}

\newcommand{\bbN}{\mathbb{N}}

\newcommand{\bbP}{\mathbb{P}}

\newcommand{\bbZ}{\mathbb{Z}}

\newcommand{\frakM}{\mathfrak{M}}

\newcommand{\frakS}{\mathfrak{S}}

\newcommand{\frakd}{\mathfrak{d}}

\newcommand{\frakg}{\mathfrak{g}}
\newcommand{\frakh}{\mathfrak{h}}

\newcommand{\frakl}{\mathfrak{l}}

\newcommand{\frako}{\mathfrak{o}}

\newcommand{\bftau}{\bfmath{\tau}}

\renewcommand{\phi}{\varphi}
\newcommand{\eps}{\varepsilon}
\renewcommand{\theta}{\vartheta}

\renewcommand{\tilde}[1]{\widetilde{#1}}

\renewcommand{\bar}[1]{\overline{#1}}

\newcommand{\id}{\mathrm{id}}
\newcommand{\Id}{\mathrm{Id}}

\renewcommand{\Im}{\mathrm{Im} \;}

\DeclareMathOperator{\Hom}{Hom}

\DeclareMathOperator{\trace}{trace}

\DeclareMathOperator{\rk}{rk}
\DeclareMathOperator{\End}{End}

\DeclareMathOperator{\Aut}{Aut}
\DeclareMathOperator{\Inn}{Inn}
\DeclareMathOperator{\Out}{Out}

\newcommand{\frakMod}{\frakM\frako\frakd}

\newcommand{\bfmath}[1]{\mbox{\boldmath $#1$}}

\newcommand{\fillwidthof}[3][c]
	{%
		\parbox
		{%
			\widthof{#2}%
		}%
		{%
			\ifx#1c%
				\centering#3%
			\else\ifx#1l%
				#3\hfill%
			\else\ifx#1r%
				\hfill#3%
			\fi\fi\fi%
		}%
	}%

\def\mylettrine#1#2 {\lettrine{#1}{#2}\space}

\newtheorem{thm}{Theorem}[section]
\newtheorem{lemma}[thm]{Lemma}
\newtheorem{prop}[thm]{Proposition}

\theoremstyle{definition}

\theoremstyle{remark}
\newtheorem{obser}[thm]{Observation}

\usepackage[top=3cm,bottom=3cm,left=3.4cm,right=3.4cm]{geometry}
\title{Geometric aspects of Iterated Matrix Multiplication}
\author{Fulvio Gesmundo}
\address{Dept. of Mathematics, Texas A{\&}M University, College Station, TX 77843-3368, U.S.A.}
\email{fulges@math.tamu.edu}

\keywords{Iterated Matrix Multiplication, symmetry group of a polynomial, dual degeneracy, Jacobian loci of hypersurfaces, degeneration of quiver orbits}

\subjclass[2010]{68Q17; 15A86, 14L40, 16G20}

\begin{document}

\begin{abstract}
This paper studies geometric properties of the Iterated Matrix Multiplication polynomial and the hypersurface that it defines. We focus on geometric aspects that may be relevant for complexity theory such as the symmetry group of the polynomial, the dual variety and the Jacobian loci of the hypersurface, that are computed with the aid of representation theory of quivers.
\end{abstract}

\maketitle

\section{Introduction} 

Let $q$ be a positive integer and let $Mat_q$ denote the vector space of $q \times q$ matrices with complex coefficients. For a positive integer $n$, we denote by $IMM_q^n$ the Iterated Matrix Multiplication polynomial, that is the polynomial on the vector space $V:=Mat_q^{\oplus n}$ of $n$-tuples of $q \times q$ matrices whose value on the $n$-tuple of matrices $(X_1 \vvirg X_n)$ is $\trace (X_n \cdots X_1)$. Thus, $IMM_q^n$ is a homogeneous polynomial of degree $n$ in $nq^2$ variables.

The main motivation for this study is the completeness of particular instances of $IMM_q^n$ for some complexity classes. For $q=3$, the sequence of polynomials $IMM_3^n$ is $\bfV\bfP_e$-complete \cite{BeCl:Alg_for_const_reg}; $\bfV\bfP_e$ is the complexity class of sequences of polynomials that admit a \emph{small formula} (see e.g. \cite[Ch. 13]{Lan:TensorBook} for details). For $q=n$, the sequence of polynomials $IMM_n^n$ is $\bfV\bfQ\bfP$-complete \cite{Bla:Compl_prob_for_qp_polynomials}; $\bfV\bfQ\bfP$ is the same complexity class for which the determinant polynomial $\det_n$ is complete; moreover, $\bfV\bfQ\bfP$ is equivalent to polynomially sized algebraic branching programs (see e.g. \cite{FoLiMaSr:bounds_depth4} and \cite{DvMaPeYe:Multi_Branching_Programs}).

We use tools from algebraic geometry and representation theory in order to study geometric properties of the polynomial $IMM^n_q$. We determine the symmetry group of the polynomial $IMM^n_q$ and we prove that this polynomial is characterized by its symmetry group. We make a study of geometric properties of the algebraic hypersurface $\calI mm^n_q\subseteq Mat_q^{\oplus n} \simeq \bbC^{nq^2}$ cut out by the polynomial $IMM^n_q$: we determine the dimension of the dual variety of $\calI mm^n_q$ and give a description of the singular locus of $\calI mm^n_q$ and of its $(n-2)$-nd Jacobian locus, namely the zero-locus of the partial derivatives of $IMM_q^n$ of order $(n-2)$.

Before we describe our goals in detail, we briefly present a possible general strategy toward the separation of complexity classes (see e.g. \cite{Lan:GCT_intro_for_geo} and \cite{Lan:Intro_to_GCT} for details). The main idea is to exploit \emph{pathologies} affecting a sequence of polynomials that is complete for a fixed complexity class, in order to prove that some given sequence of polynomials, not sharing such pathology, does not belong to the complexity class. More precisely, if $g_n \in S^{d_n}\bbC^{N_n}$ is a sequence of polynomials that is complete for a complexity class $\calC$, then a sequence of polynomials $f_m \in S^{e_m} \bbC^{M_m}$ is in $\calC$ if and only if it can be polynomially reduced to $g_n$, namely if and only if there is a polynomial function $n(m)$ such that, for every $m$
\[
 z^{d_{n(m)}-e_{m}} f_m \in \End(\bbC^{N_{n(m)}}) \cdot g_{n(m)} \subseteq S^{d_{n(m)}} (\bbC^{N_{n(m)}}),
\]
where $z$ is a \emph{padding variable} and $\bbC^{M_m} \oplus \bbC z$ is viewed as a subspace of $\bbC^{N_{n(m)}}$. When we say \emph{pathology}, we mean a geometric property that is shared by all polynomials in the set $\End(\bbC^{N_n}) \cdot g_n$ (and possibly other polynomials) but it is not shared by the padded polynomials $z^{d_n -e_m} f_m$, whenever $n$ grows at most polynomially in $m$; determining such property would prove that the sequence $\{ f_m \}$ does not belong to the complexity class $\calC$.

The Geometric Complexity Theory (GCT) program (see \cite{GCT2}) focuses on the study of polynomials that are \emph{characterized} by their \emph{symmetry group}, that is the stabilizer under the action of the general linear group of the space generated by the variables. If $f \in S^d W$ and $G_f \subseteq GL(W)$ is its symmetry group, then we say that $f$ is characterized by $G_f$ if it is the only polynomial, up to scale, whose stabilizer contains $G_f$. The algebraic Peter-Weyl Theorem (see e.g. \cite[Ch. 6, Sec. 2.6]{Pro:LieGroups}) leads to a description of the ring of regular functions on the group orbit $GL(W)\cdot f \subseteq S^d W$ in terms of $G_f$-invariants; if $f$ is characterized by its stabilizer, then the coordinate ring of the orbit of $f$ is unique as $GL(W)$-module among the coordinate rings of $GL(W)$-orbits in $S^d W$. In particular, sequences of polynomials that are complete for some complexity class and that are characterized by their stabilizer can be considered \emph{best} representatives of their class. This approach was used in the past to provide representation theoretic obstructions proving lower bounds on the complexity of certain polynomials (see \cite{BLMW} and \cite{BuIk:Expl_Bounds_Via_GCT}) and will be part of future work in the setting of $IMM^n_q$.

In Section 2, we determine the symmetry group of the polynomial $IMM_q^n$ (Thm. \ref{thm: stabilizer of IMM}). Furthermore, we prove that $IMM_q^n$ is characterized by its stabilizer (Prop. \ref{prop: IMM characterized by its stabilizer}). These results motivate the study of geometric properties of $IMM_q^n$, that we develop in the rest of the paper.

In Section 3 we prove that the dual variety of the hypersurface $\calI mm_q^n$ is itself a hypersurface (Thm. \ref{thm: dual of calImm is hypersurface}); in \cite{LaMaRa:Deg_duals}, dual degeneracy was shown to be a property that can lead to lower bounds on the complexity of polynomials, and our result shows that unfortunately it cannot be exploited in our setting. 

In Section 4, we study the singular locus of $\calI mm^n_q$, namely the zero locus of its first order partial derivatives. The Representation Theory of Quivers provides powerful tools to describe the irreducible components of this singular locus. We characterize its irreducible components in terms of particular nilpotent representations for the Euclidean equioriented quiver $\tilde{\bfA}_{n}$ (Thm. \ref{thm: sing compt iff rk-maximal}) and we give a formula for their dimensions (Thm. \ref{thm: desingularization of irred compts of Sing}). We close the section with some explicit examples of the decomposition of the singular locus into its irreducible components, for small values of $q$ and $n$.

In Section 5, we determine the dimension and the number of irreducible components of the $(n-2)$-nd Jacobian locus of $\calI mm^n_q$, namely the zero locus of the partial derivatives of order $(n-2)$ (Thm. \ref{thm: n-2 singular locus}).

Since the paper deals with several aspects of algebraic geometry and representation theory, we occasionally recall the main definitions and some basic results concerning the topics that we discuss, in order to make the paper accessible to a broader audience.

\subsection*{Notation} We set the notation that will be used throughout the paper. 

Let $U_1 \vvirg U_n$ be complex vector spaces of dimension $q$. It will be useful to read the index $\alpha$ of a space $U_\alpha$ modulo $n$; in particular, $U_{n+1} = U_1$. For $\alpha = 1 \vvirg n$, let $A_\alpha := U_\alpha^* \otimes U_{\alpha +1}$ and let $V := A_1^* \ooplus A_n^*$. Note that $\dim V = nq^2$.

The Iterated Matrix Multiplication polynomial in this setting is $IMM_q^n \in S^n V$, a homogeneous polynomial of degree $n$ on $V^*$ defined by $IMM_q^n (X_1 \vvirg X_n) = \trace(X_n \cdots X_1)$ where, for any $\alpha$, $X_\alpha \in A_\alpha$ is a map $X_\alpha : U_\alpha \to U_{\alpha+1}$ and $\cdot$ indicates the composition of linear maps (or equivalently matrix multiplication, in coordinates).

Notice that $S^nV$ has the following multidegree decomposition
\[
 S^n V \simeq \bigoplus_{\substack{a_1 + \dots + a_n = n \\ a_i \geq 0}} S^{a_1} A_1^* \ootimes S^{a_n} A_n^*;
\]
$IMM_q^n$ belongs to the component $A_1^* \ootimes A_n^*$ of multidegree $(1\vvirg 1)$.

We will denote by $\calI mm _q^n$ either the algebraic hypersurface $V(IMM^n_q) \subseteq \bbP V^*$ (in Section 3) or its affine cone in $V^*$ (in Sections 4 and 5); we will specify which one is considered at the beginning of each section.

Besides the coordinate free definition, it is convenient to have an explicit description of $IMM^n_q$ in terms of bases of the $U_\alpha$'s and of their duals.

For $\alpha = 1 \vvirg n$, let $\{(u_\alpha)^i\}_{i=1 \vvirg q}$ be a basis for $U_\alpha$ and let $\{(\eta _{\alpha})_{j}\}_{j=1 \vvirg q}$ be its dual basis in $U^{*}_\alpha$. Let $(x_\alpha)^i_j$ be the basis of $A_\alpha^*$ defined by $(x_\alpha)^i_j := (u_\alpha)^i \otimes (\eta_{\alpha+1})_j$ and let $(\xi_{\alpha})_i^j := (\eta_\alpha)_i \otimes (u_{\alpha+1})^j$ be the corresponding dual basis of $A_\alpha$. In coordinates, the element $(x_\alpha)^i_j$ can be identified with the linear map sending a matrix $X_\alpha \in A_\alpha$ to the value of its $(i,j)$-th entry and $(\xi_\alpha)^j_i$ can be identified with the matrix having $1$ at the $(i,j)$-th entry and $0$ elsewhere.

With this notation, we have 
\[
 IMM_q^n = \sum_{\ell_1 \vvirg \ell_n =1}^q (x_n)^{\ell_1}_{\ell_{n}} (x_{n-1})^{\ell_{n}}_{\ell_{n-1}} \cdots (x_1)^{\ell_2}_{\ell_1} \in S^n (A_1^* \ooplus A_n^*).
\]
 
 \section{The symmetry group of $IMM^n_q$}

The symmetry group of a polynomial $f \in S^d W$ is the stabilizer of $f$ under the action of $GL(W)$ on $S^dW$. Let $\calS \subseteq GL(V)$ denote the symmetry group of $IMM^n_q \in S^n V$, and let $\calS_0$ be the connected component of the identity in $\calS$.

The goal of this section is to determine the group $\calS$ and to prove that $IMM_q^n$ is characterized by its symmetries, namely that if $f \in S^n V$ is a polynomial whose symmetry group contains $\calS$, then $f$ is a scalar multiple of $IMM^n_q$.

If $G$ is a group and $g,h \in G$, let $h^g$ denote the conjugate of $h$ by $g$, namely $h^g := g h g^{-1}$; similarly, if $H$ is a subgroup of $G$, let $H^g$ denote the conjugate subgroup, namely $H^g := \{ h^g : h \in H\}$. $N_G(H)$ denotes the normalizer of $H$ in $G$, namely $N_G(H) := \{ g \in G: H^g = H\}$. $C_G(H)$ denotes the centralizer of $H$ in $G$, namely $C_G(H) := \{g \in G : h^g = h , \ \forall h \in H \}$.

Let $\Aut(G)$ denote the group of $\bbC$-linear automorphisms of $G$. $\Inn(G)$ denotes the subgroup of inner automorphisms and $\Out(G) := \Aut(G) / \Inn(G)$ is the group of outer automorphisms.

 We record the following standard fact:
\begin{lemma}
 Let $f \in S^d W$ be a polynomial and let $G_f \subseteq GL(W)$ be its stabilizer under the action of $GL(W)$ on $S^d W$. Let $G_f^0$ be the connected component of the identity in $G_f$. Then $G_f \subseteq N_{GL(W)} (G_f^0)$.
 \begin{proof}
  For any $h \in G_f^0$ and any $g \in G_f$ we want to show that $h^g \in G_f^0$. Let $\alpha : [0,1] \to G_f^0$ be a curve such that $\alpha(0) = I_W$ (the identity of $GL(W)$) and $\alpha(1) = h$. Consider $\beta(t) = \alpha(t)^g$. By continuity, the curve $\beta(t)$ satisfies $\beta(0) = I_W$ and $\beta(1) = h^g$.
  
  Moreover, for any $t$, $\beta(t) \cdot f = g \alpha(t) g^{-1} \cdot f = f$ since $g,\alpha(t) \in G_f$. Therefore $\beta(t) \in G_f$ is a curve in $G_f$ from $I_W$ to $h^g$. By continuity $\beta(t) \in G_f^0$ and therefore $h^g \in G_f^0$.
 \end{proof}
\end{lemma}

In particular, $\calS \subseteq N_{GL(V)} \calS_0$.

Our first goal is to determine $\calS_0$. This can be done by a calculation involving Lie algebras:
\begin{obser}\label{obs: lie algebra of stab kills the poly}
Let $f \in S^d W$ be a polynomial and let $G$ be a connected Lie group acting on $W$. Let $\frakg$ be the Lie algebra of $G$ and let $\frakh := \{ X \in \frakg : X.f = 0\}$. Then $\frakh$ is the Lie algebra of a subgroup $H \subseteq G$ and $H$ is the connected component of the stabilizer of $f$ in $G$.
\end{obser}

For any $\alpha$, $GL(U_\alpha)$ embeds in $GL(V)$ via a homomorphism $\phi_\alpha$ given by the following composition:
\[
 GL(U_{\alpha}) \xto{\bar{\phi}_\alpha} GL(A_{\alpha-1} \oplus A_{\alpha}) \simeq GL(A_{\alpha-1}^* \oplus A_{\alpha}^*) \to GL(V)
\]
where for $g \in GL(U_{\alpha}) $, $\bar{\phi}_\alpha(g) : (X_{\alpha-1}, X_\alpha) \mapsto (g^{-1}X_{\alpha-1},X_\alpha g)$, the middle map is the natural isomorphism between $GL(W)$ and $GL(W^*)$ (given by the transpose inverse) and the last map is a canonical embedding.

This defines a group homomorphism $ \Phi : GL(U_1) \ttimes GL(U_n) \to GL(V)$, whose kernel is $\ker \Phi = \bbC^* (\Id_{U_1} \ttimes \Id_{U_n})$, so that $\Im \Phi \simeq (GL(U_1) \ttimes GL(U_n))/\bbC^*$ is a subgroup of $GL(V)$.

\begin{prop}
 The image of the homomorphism $\Phi$ is $\calS_0$.
 \end{prop}
 \begin{proof}
The Lie algebra of the image of $\Phi$ coincides with the image of the differential $d \Phi$ of $\Phi$; we will determine this Lie algebra and we will show that it is the annihilator of $IMM_q^n$ in $\frakg \frakl (V)$. For every $\alpha$, we have $\frakg \frakl (U_\alpha) = U_{\alpha}^* \otimes U_\alpha$. Moreover, $\frakg\frakl (V) = V^* \otimes V = \bigoplus_{\beta,\gamma} A_{\beta}\otimes A_{\gamma}^*$.

The image of the differentials $d \phi_\alpha$ is given by

\begin{equation*}
\begin{aligned}
d\phi_\alpha : U_\alpha^* \otimes U_\alpha &\to (A_{\alpha-1} \otimes A_{\alpha-1}^* ) \oplus (A_{\alpha} \otimes A_{\alpha}^*)
\\
L &\mapsto id_{U_{\alpha-1}} \otimes (-L^T) + (L) \otimes id_{U_{\alpha+1}^*},
\end{aligned}
\end{equation*}
where, under the reordering homomorphism, the image is viewed as an element of 
\begin{equation*}
\begin{aligned}
&(U_{\alpha-1}^* \otimes U_\alpha \otimes U_{\alpha-1} \otimes U_{\alpha}^*) \oplus (U_{\alpha}^* \otimes U_{\alpha+1} \otimes U_{\alpha} \otimes U_{\alpha+1}^*) \simeq \\ \simeq &(U_{\alpha-1}^*\otimes U_{\alpha-1} \otimes U_\alpha  \otimes U_{\alpha}^*) \oplus (U_{\alpha}^* \otimes U_{\alpha} \otimes U_{\alpha+1}  \otimes U_{\alpha+1}^*).
\end{aligned}
\end{equation*}
The differential $d\Phi$ is given by the sum of the differentials $d\phi_\alpha$.

In coordinates, on a basis vector $(\eta_\alpha)_i \otimes (u_\alpha)^j$, we have
\begin{align*}
 d\phi_\alpha : (\eta_\alpha)_i \otimes (u_\alpha)^j \mapsto  -&\textsum_{s_{\alpha-1}} (\xi_{\alpha-1})^j_{s_{a-1}} \otimes (x_{\alpha-1})^{s_{\alpha-1}}_{i} +\\
 &+\textsum_{s_{\alpha+1}} (\xi_{\alpha})^{s_{\alpha+1}}_{i} \otimes (x_{\alpha})^{j}_{s_{\alpha+1}}.
\end{align*}
It is immediate that $U_\alpha^* \otimes U_\alpha$ annihilates $IMM_q^n$ (under this embedding) and therefore the image of $d\Phi$ is contained in the Lie algebra annihilating $IMM_q^n$.

In order to prove the other containment, we need to show that if $L \in \frakg\frakl(V)$ annihilates $IMM^n_q$, then it belongs to the image of $d\Phi$. 

Suppose $L$ annihilates $IMM^n_q$. $L$ cannot have a non-trivial component $\tilde{L}$ in $A_\beta \otimes A_\gamma^*$ with $\beta \neq \gamma$. The component $\tilde{L}$ would generate in $L\cdot IMM_q^n$ a non-zero term involving monomials that are quadratic in the variables from $A_\gamma^*$ and with no variable from $A_\beta^*$; such terms can only arise from components in $A_\beta \otimes A_\gamma^*$. Therefore, $\tilde{L}$ has to annihilate $IMM^n_q$; but it is immediate that the images of $IMM^n_q$ under distinct monomials in $A_\beta \otimes A_\gamma^*$ are linearly independent, and this proves $\tilde{L} = 0$.

This shows that the annihilator of $IMM^n_q$ in $\frakg \frakl (V)$ is contained in $\bigoplus_\alpha A_\alpha \otimes A_\alpha^*$.

A similar, slightly more involved, calculation shows that the annihilator coincides with the image of $d\Phi$.

By Observation \ref{obs: lie algebra of stab kills the poly}, we conclude that $\Im \Phi = \calS_0$.
 \end{proof}

 Our next goal is to determine $N_{GL(V)} (\calS_0)$. Notice that every element of $N_{GL(V)} (\calS_0)$ induces, by conjugation, a $\bbC$-linear automorphism on $\calS_0$. Moreover, in our case, the quotient $N_{GL(V)}(\calS_0)/\calS_0$ injects into $\Out(\calS _0)$:
 
\begin{lemma}\label{lemma: uniqueness of inducing element}
 Let $g_1,g_2 \in N_{GL(V)}(\calS_0) \subseteq GL(V)$ be two elements realizing (by conjugation) the same automorphism $\phi \in \Aut(\calS_0)$ on $\calS_0$. Then $g_1^{-1}g_2 \in \Phi(\bbC^* \Id_{U_1} \ttimes \bbC^*\Id_{U_n})$. In particular there exists an injective homomorphism $N_{GL(V)}(\calS_0)/\calS_0 \to \Out(\calS_0)$.
 \begin{proof}
  It suffices to prove that the centralizer $C_{GL(V)}(\calS_0)$ of $\calS_0$ in $GL(V)$ is $\Phi(\bbC^* \Id_{U_1} \ttimes \bbC^*\Id_{U_n})$. Let $c \in C_{GL(V)}(\calS_0)$. Then for every $\alpha$, conjugation by $c$ acts trivially on the copy of $GL(U_\alpha)$ embedded in $GL(V)$. In other words, $c$ defines a $GL(U_\alpha)$-equivariant map from $U_\alpha$ to $U_\alpha$; by Schur's Lemma, such map has to be the identity up to scale.
  
  The second part of the statement follows by the fact that $\Phi(\bbC^* \Id_{U_1} \ttimes \bbC^*\Id_{U_n}) \subseteq \calS_0$.
 \end{proof}
\end{lemma}

The group of outer automorphisms of $\calS_0$ can be characterized in terms of graph automorphisms of its Dynkin diagram:

\begin{prop}[see e.g. \cite{FulHar:RepTh} - Proposition D40]\label{prop: fultonharris - D40}
Let $G$ be a complex semisimple Lie group and let $\Delta$ be its Dynkin diagram. Then the group of outer $\bbC$-linear automorphisms of $G$ is isomorphic to the group of graph automorphisms $\Delta$. Moreover, every automorphism obtained in this way is defined by (the lift to $G$ of) the Lie algebra automorphism obtained by the permutation of the roots induced by the corresponding graph automorphism on $\Delta$.
\end{prop}

Let $G$ be a complex semisimple Lie group with Dynkin diagram $\Delta$, and let $\lambda$ be a weight for $G$. We denote with $\Delta(\lambda)$ the \emph{marked Dynkin diagram} of the irreducible representation of $G$ of highest weight $\lambda$, namely the Dynkin diagram $\Delta$ where the vertex corresponding to the fundamental weight $w$ is labeled by the integer coefficient of $w$ in the expression of $\lambda$ as sum of the fundamental weights; notice that the group of graph automorphisms $\Aut(\Delta)$ has a natural action on the set $\Lambda_G := \{ \Delta(\lambda) : \lambda \text{ weight}\}$ of all possible marked Dynkin diagram of $G$. The following result will allow us to determine the subgroup of $\Out(\calS_0)$ of automorphisms that can be realized as conjugation by elements of $GL(V)$:

\begin{prop}[\cite{BeGaLa:Linear_preservers}, Prop. 2.2 and Cor. 2.4]\label{prop: garibaldi}
 Let $\rho: G \to GL(W)$ be a representation of $G$. Suppose $W = W_1 \ooplus W_k$ with $W_j$ distinct irreducible representations of $G$ with $\lambda_j = hw(W_j)$. Let $\phi \in \Out(G) \simeq \Aut(\Delta)$ be an automorphism that can be realized via conjugation by an element of $GL(W)$. Then $\phi$, viewed as an element of $\Aut(\Delta)$, stabilizes the set $\{\Delta(\lambda_j) : j =1 \vvirg k\} \subseteq \Lambda$.
\end{prop}

If $G,H$ are two groups, with $H \subseteq \frakS_d$ for some positive integer $d$, let $G \wr H$ denote the wreath product given by the action of $H$ on $d$ copies of $G$, namely $G \wr H := (\prod_1^d G) \rtimes H$, where $H$ acts by permutations on the copies of $G$ appearing in the direct product. We classify the group of graph automorphisms of the Dynkin diagram of $\calS_0$:

\begin{lemma}
 We have $\Out(\calS_0) \simeq \bbZ_2 \wr \frakS_n$.
 \begin{proof}
The Dynkin diagram $\Delta$ of $\calS_0$ is
\[
\begin{tikzpicture}
	\draw (0, 3) -- (1,3);
	\draw[dashed] (1, 3) -- (4,3);
	\draw (4, 3) -- (5,3);
	\draw[fill=white] (0,3) circle(.15);
	\draw[fill=white] (1,3) circle(.15);
	\draw[fill=white] (4,3) circle(.15);
	\draw[fill=white] (5,3) circle(.15);
	
	\draw[yshift=.3cm] (0,3) node{$~_1$};
	\draw[yshift=.3cm] (1,3) node{$~_2$};
	\draw[yshift=.3cm] (4,3) node{$~_{q-2}$};
	\draw[yshift=.3cm] (5,3) node{$~_{q-1}$};

	\draw (0, 2) -- (1,2);
	\draw[dashed] (1, 2) -- (4,2);
	\draw (4, 2) -- (5,2);
	\draw[fill=white] (0,2) circle(.15);
	\draw[fill=white] (1,2) circle(.15);
	\draw[fill=white] (4,2) circle(.15);
	\draw[fill=white] (5,2) circle(.15);

	\draw[yshift=.3cm] (0,2) node{$~_1$};
	\draw[yshift=.3cm] (1,2) node{$~_2$};
	\draw[yshift=.3cm] (4,2) node{$~_{q-2}$};
	\draw[yshift=.3cm] (5,2) node{$~_{q-1}$};
	
	\draw (2.5,1) node{$\vdots$};
	
	\draw (0, 0) -- (1,0);
	\draw[dashed] (1, 0) -- (4,0);
	\draw (4,0) -- (5,0);
	
	\draw[fill=white] (0,0) circle(.15);
	\draw[fill=white] (1,0) circle(.15);
	\draw[fill=white] (4,0) circle(.15);
	\draw[fill=white] (5,0) circle(.15);
	
	\draw[yshift=.3cm] (0,0) node{$~_1$};
	\draw[yshift=.3cm] (1,0) node{$~_2$};
	\draw[yshift=.3cm] (4,0) node{$~_{q-2}$};
	\draw[yshift=.3cm] (5,0) node{$~_{q-1}$};
	
	\draw (-.8,3) node{$~_{\bfA_{q-1}^{(1)}:}$};
	\draw (-.8,2) node{$~_{\bfA_{q-1}^{(2)}:}$};
	\draw (-.8,0) node{$~_{\bfA_{q-1}^{(n)}:}$};
%
%
\end{tikzpicture}
\]
where the $\alpha$-th diagram $\bfA _{q-1}^{(\alpha)}$ corresponds to the embedded copy of $GL(U_\alpha)$.

The permutation group $\frakS_n$ acts on $\Delta$ by permuting the copies of $\bfA_{q-1}^{(\alpha)}$. Moreover, each $\bfA_{q-1}^{(\alpha)}$ has a copy of $\bbZ_2$ acting on it by reversing the order of the vertices. Let $\bbZ_2^{(\alpha)}$ denote the copy acting on $\bfA_{q-1}^{(\alpha)}$. Therefore $\Aut(\Delta) = \left(\prod_{\alpha = 1}^n \bbZ_2^{(\alpha)}\right) \rtimes \frakS_n = \bbZ_2 \wr \frakS_n$. We conclude, by Proposition \ref{prop: fultonharris - D40}.
 \end{proof}
\end{lemma}

\begin{obser}
 We have $V = A_1^* \ooplus A_n^*$ and $A_\alpha = U_\alpha^* \otimes U_{\alpha + 1}$. The irreducible representations of $\calS_0$ occurring in $V$ are therefore $U_\alpha^* \otimes U_{\alpha + 1}$ for every $\alpha$. The marked Dynkin diagram of $A_\alpha$ is
 
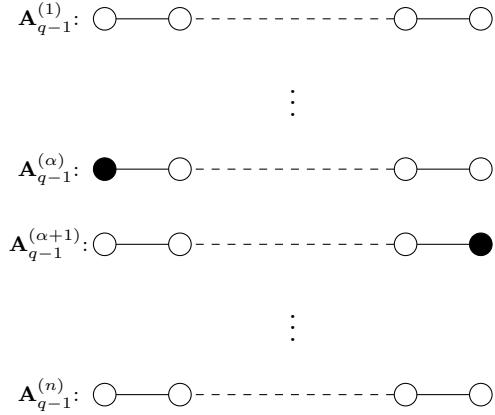
\[ \begin{tikzpicture}
	\draw (0,5) -- (1,5);
	\draw[dashed] (1,5) -- (4,5);
	\draw (4,5) -- (5,5);
	\draw[fill=white] (0,5) circle(.15);
	\draw[fill=white] (1,5) circle(.15);
	\draw[fill=white] (4,5) circle(.15);
	\draw[fill=white] (5,5) circle(.15);

	\draw (2.5,4) node{$\vdots$};

  	\draw (0, 3) -- (1,3);
	\draw[dashed] (1, 3) -- (4,3);
	\draw (4, 3) -- (5,3);
	\draw[fill=black] (0,3) circle(.15);
	\draw[fill=white] (1,3) circle(.15);
	\draw[fill=white] (4,3) circle(.15);
	\draw[fill=white] (5,3) circle(.15);
	
	\draw (0, 2) -- (1,2);
	\draw[dashed] (1, 2) -- (4,2);
	\draw (4, 2) -- (5,2);
	\draw[fill=white] (0,2) circle(.15);
	\draw[fill=white] (1,2) circle(.15);
	\draw[fill=white] (4,2) circle(.15);
	\draw[fill=black] (5,2) circle(.15);

	\draw (2.5,1) node{$\vdots$};
	
	\draw (0, 0) -- (1,0);
	\draw[dashed] (1, 0) -- (4,0);
	\draw (4,0) -- (5,0);
	
	\draw[fill=white] (0,0) circle(.15);
	\draw[fill=white] (1,0) circle(.15);
	\draw[fill=white] (4,0) circle(.15);
	\draw[fill=white] (5,0) circle(.15);

	\draw (-.8,5) node{$~_{\bfA_{q-1}^{(1)}:}$};
	\draw (-.8,3) node{$~_{\bfA_{q-1}^{(\alpha)}:}$};
	\draw (-.8,2) node{$~_{\bfA_{q-1}^{(\alpha+1)}:}$};
	\draw (-.8,0) node{$~_{\bfA_{q-1}^{(n)}:}$};
	
%
%
\end{tikzpicture}
 \]
where the black filling means that the vertex is labeled with a $1$.
\end{obser}

We can now completely characterize the stabilizer $\calS$ of $IMM_q^n$.
\begin{thm}\label{thm: stabilizer of IMM}
 The stabilizer $\calS$ of $IMM_q^n$ in $GL(V)$ is isomorphic to $\calS_0 \rtimes D_n$, where $D_n$ is the dihedral group of the regular $n$-gon. In particular $D_n = \bbZ_n \rtimes \bbZ_2$, where $\bbZ_n \subseteq \frakS_n$ is the cyclic group generated by the $n$-cycle $(1\vvirg n)$ acting by permuting the $X_\alpha$'s and $\bbZ_2 = \langle \tau \rangle$ acts by $\tau(X_1 \vvirg X_n) = (X_n^T\vvirg X_1^T)$.
\end{thm}
\begin{proof}
Fix bases for each of the $U_\alpha$'s and the corresponding dual bases of $U_\alpha^*$'s. Define isomorphisms $\phi_{\beta\alpha} : GL(U_\alpha) \to GL(U_\beta)$ identifying the bases of $U_\alpha$ and $U_\beta$: in particular, these isomorphisms satisfy a cocycle type condition $\phi_{\alpha\gamma}\circ\phi_{\gamma\beta}\circ \phi_{\beta\alpha} = id_{U_\alpha}$ for every $\alpha,\beta,\gamma$. Moreover, let $\delta_\alpha : U _\alpha \to U_\alpha^*$ be the isomorphism defined by identifying the basis of $U_\alpha$ with its dual (so in particular $(\delta^{-1})^T = \delta$).

For every $\alpha$ define $\tau_\alpha : GL(U_\alpha) \to GL(U_\alpha)$ by $\tau_\alpha(g_\alpha): u \mapsto \delta_\alpha^{-1}\circ(g_\alpha^T)^{-1}\circ \delta_\alpha (u)$, where $(g_\alpha^T)^{-1}$ is canonically an element of $GL(U_\alpha^*)$. It is clear that $\tau_\alpha$ defines an automorphism of $GL(U_\alpha)$ that is outer whenever $q > 2$ and has order $2$. In particular, for every $\alpha$, $\Out(GL(U_\alpha)) \simeq \Aut(\bfA_{q-1}^{(\alpha)}) \simeq \bbZ_2^{(\alpha)} \simeq \langle \tau_\alpha \rangle$.

Moreover, the group $\frakS_n \subseteq \Out(\calS_0)$ can be viewed as acting by the isomorphisms $\phi_{\alpha\beta}$. We obtain
\[
 \Out(\calS_0) = \left(\prod_{\alpha = 1}^{n} \langle \tau_\alpha \rangle \right) \rtimes \frakS_n.
\]
We want to determine the subgroup of $\Out(\calS_0)$ fixing (as a set) the set of marked Dynkin diagrams $ \Sigma_\Delta = \{\Delta(A_\alpha)\}$. Notice that the cycle $\rho = (1 \vvirg n) \in \frakS_n \subseteq \Out(\calS_0)$ sends $\Delta(A_\alpha)$ in $\Delta(A_{\alpha+1})$ for every $\alpha$. Therefore it stabilizes $\Sigma_\Delta$. Moreover, the element of $\left(\textprod_{1}^{n} \langle \tau_\alpha \rangle \right) \rtimes \frakS_n$ defined by
\begin{align*}
 \bftau = \left((\tau_1 \vvirg \tau_n), \left(\begin{array}{cccc}
                                             1 & 2 & \cdots & n \\
                                             n & n-1 & \cdots & 1
                                            \end{array}
 \right)\right)
\end{align*}
sends $\Delta(A_\alpha)$ to $\Delta(A_{n+1-\alpha})$, so it stabilizes $\Sigma_\Delta$. Now, suppose that $\psi$ is an outer automorphism stabilizing $\Sigma_\Delta$. By the action of $\langle \rho \rangle$ we may assume that it fixes $\Delta(A_1)$. There are two possibilities: either the marked vertices of $\Delta(A_1)$ are fixed, or they are swapped. If they are swapped by $\psi$, we can act by $\bftau$ and then again by $\langle \rho \rangle$ and assume that they are fixed. This means that the first vertex of $\bfA_{q-1}^{(2)}$ is fixed and therefore the entire $\bfA_{q-1}^{(2)}$ is fixed. In particular $\Delta(A_2)$ is fixed. Similarly, every $\Delta(A_\alpha)$ is fixed by $\psi$ after possibly acting with $\rho$ and $\bftau$. This argument shows that $\psi$ is generated by $\bftau$ and $\rho$. Notice that conjugation by $\bftau$ induces the inversion on $\langle \rho \rangle$. This shows that the subgroup of $\Out(\calS_0)$ that fixes $\Sigma_\Delta$ as a set is isomorphic to $D_n$, the dihedral group 
of the regular $n$-gon.

From Proposition \ref{prop: garibaldi}, we deduce that $N_{GL(V)}\calS_0\simeq \calS_0 \rtimes K$, where $K$ is a subgroup of $D_n$. However, the elements of $GL(V)$ defined by
\begin{align*}
 \rho: (X_1 \vvirg X_n) &\mapsto (X_2 \vvirg X_n,X_1), \\
 \tau: (X_1 \vvirg X_n) &\mapsto (X_n^T \vvirg X_1^T),
\end{align*}
clearly lie in $N_{GL(V)}\calS_0$ and generate a copy of $D_n$. We conclude that $N_{GL(V)}\calS_0\simeq \calS_0 \rtimes D_n$.

Thus, $\calS = N_{GL(V)}\calS_0$, since the elements $\rho,\tau$ defined above stabilize the polynomial $IMM_q^n$.
\end{proof}

We conclude this section with the proof that the subgroup $\calS_0$ of $\calS$ characterizes the polynomial $IMM_q^n$.

 \begin{prop}\label{prop: IMM characterized by its stabilizer}
  Let $P \in S^n V$ be a polynomial that is stabilized by $\calS_0$. Then $P = c IMM_q^n$ for some $c \in \bbC$.
  \begin{proof}
   Consider the decomposition of $S^nV$ as $\calS_0$-module
   \[
    S^n V = \bigoplus_{a_1 + \cdots + a_n = n} S^{a_1} A_1^*  \ootimes S^{a_n}A_n^* .
   \]
Write $P = \sum_{\vert \bfa \vert = n} P_\bfa$ in accordance with its multidegree. $\calS_0$ acts on each component. Therefore, if $P$ is stabilized by $\calS_0$, then, each $P_\bfa$ is stabilized by $\calS_0$.

We will show that $P_{(1^n)} = c IMM_q^n$ and that $P_\bfa = 0$ if $\bfa \neq (1^n)$.

The first claim follows from the fact that $IMM_q^n$ coincides, as a tensor in $A_1 \ootimes A_n$, with $\id_{U_1} \ootimes \id_{U_n}$, that is the only tensor in $A_1 \ootimes A_n$ that is stabilized by $\calS_0$.

In order to prove the second claim, consider the image via $\Phi : GL(U_1) \ttimes GL(U_n) \to GL(V)$ of $\bbC^* \id_{U_1} \ttimes \bbC^* \id_{U_n}$, that is $\{z_1 \id_{A_1} \ttimes z_n \id_{A_n} : z_1 \cdots z_n =1\}$. If $\bfa \neq (1^q)$, let $\beta_j$ for $j=1 \vvirg k$ be the indices for which $a_{\beta_j} = 0$; any element of the form $z_1 \id_{A_1} \ttimes z_n \id_{A_n}$ with $z_{\beta_1} \cdots z_{\beta_k} \neq 1$ does not stabilize $P_\bfa$ if $P_\bfa \neq 0$. This concludes the proof.
  \end{proof}
 \end{prop}

\section{The dual variety of $\calI mm_q^n$}

We refer to \cite[Ch. 1]{GKZ} for basic facts on dual varieties.

If $f \in S^d W^*$ is a general polynomial of degree $d$ on $W$, then the dual variety $V(f)^\vee$ of the hypersurface $V(f) \subseteq \bbP W$ is itself a hypersurface; we say that $V(f)$ is dual degenerate if $V(f)^\vee$ is not a hypersurface. In particular, dual degeneracy of the variety of singular $n\times n$ matrices (the zero set of the $n\times n$ determinant polynomial) was used in \cite{LaMaRa:Deg_duals} to provide a lower bounds on the determinantal complexity of the permanent polynomial.

The goal of this section is to prove that the dual variety of the projective variety $\calI mm_q^n \subseteq \bbP V^*$ is a hypersurface for every $n \geq 1$ and every $q \geq 2$. The main result that we will need is B. Segre's dimension formula:
\begin{thm}[B. Segre dimension formula, \cite{GKZ}, Ch. 1, Thm. 5.3]
Let $f \in S^d W^*$ be an irreducible polynomial. Let $Y = V(f) \subseteq \bbP W$ be the zero locus of $f$ and let $Y ^\vee \subseteq \bbP W^*$ be the dual variety of $Y$. Then
\[
 \dim Y^\vee = \rk (H_f (y)) - 2
\]
where $H_{f}(y)$ is the Hessian matrix of $f$ evaluated at a general point $y \in Y$.
\end{thm}

We will find a point of $\calI mm_q^n$ where the Hessian matrix of $IMM_q^n$ is non-singular. Semicontinuity of matrix rank provides that the Hessian matrix at the general point is non-singular and therefore the dual variety $(\calI mm_q^n)^\vee$ is a hypersurface in $\bbP V$.

The second order partials of $IMM_q^n$ are the following:
\begin{align*}
 &\frac{\partial^2}{\partial(x_\beta)^r_s \partial (x_\alpha)^j_k } IMM^n_q =\\
 &=\left\{ \begin{array}{ll}
 0 & \text{if } \alpha = \beta ;\\
 0 & \text{if } \beta = \alpha-1 \text{ and } k \neq r \\
 0 & \text{if } \beta = \alpha+1 \text{ and } s \neq j;\\
 (X_{\alpha -2} \cdots X_1X_n \cdots X_{\alpha+1} )_j^s & \text{if } \beta = \alpha -1 \text{ and } k = r;\\
 (X_{\alpha -1} \cdots X_1 X_n \cdots X_{\alpha+2} )_r^k & \text{if } \beta = \alpha +1 \text{ and } s = j;\\
 (X_{\alpha - 1} \cdots X_{\beta +1})^k_r (X_{\beta-1} \cdots X_{\alpha +1})^s_j & \text{if } \beta \neq \alpha,\alpha \pm 1. \\
 \end{array}\right.
\end{align*}

We denote by $\nabla_\bullet$ (resp. $\nabla^\bullet$) the row (resp. column) vector of first order differential operators in the variables $(x_\alpha)^j_k$ ordered according to the lexicographical order on $(\alpha,j,k)$ (resp. on $(\alpha,k,j)$). The Hessian matrix of a polynomial $f \in Sym(V)$ is $\nabla^\bullet \nabla_\bullet f$; if $S$ is a subset of the set of variables, we denote by $\nabla_S$ (resp. $\nabla^S$) the sub-vector of $\nabla_\bullet$ (resp. $\nabla^\bullet$) of the differential operators in the variables of $S$; let $\nabla_\alpha = \nabla_{\{(x_\alpha)^i_j \}_{i,j}}$ and similarly $\nabla^\alpha$. The choice of different orderings for the differential operators in $\nabla_\bullet$ and $\nabla^\bullet$ breaks the usual symmetry of the Hessian matrix, but in our case, exploiting the symmetries of $IMM_q^n$ and of the point in $\calI mm^n_q$ that we will choose, gives to the Hessian a block structure that will be extremely convenient for the calculation.

Let $H := \nabla^\bullet \nabla_\bullet IMM^n_q$ be the $nq^2 \times nq^2$ Hessian matrix of $IMM^n_q$. It is an $n\times n$ block matrix consisting of $q^2 \times q^2$ blocks; the $(\alpha,\beta)$-th block is $H^\alpha_\beta = \nabla^{\alpha} \nabla_{\beta}  IMM^n_q$. 

Let $\omega$ be a root of the univariate polynomial $\phi(t) = t^n+(q-1)$ and consider the point $p = (X_\alpha)_{\alpha = 1\vvirg n} \in \bbP V^*$ where $X_\alpha = X = (x^i_j)$, is the diagonal matrix with $x^i_i = 1$ for $i\leq q-1$, $x^q_q = \omega$. We have
\[
IMM^n_q (p) = \trace(X^n) = \underbrace{1 + \cdots + 1}_{q-1} + \omega^n = \phi(\omega) = 0,
\]
so $p \in \calI mm^n_q$. 

We denote by $H^\alpha_\beta(p)$ the blocks of the Hessian matrix at the point $p$. Each block $H^{\alpha}_\beta(p)$ is a diagonal matrix and the symmetries of $IMM^n_q$ and of the point $p$ provide that $H^\alpha_\beta(p) = H^1_{\beta-\alpha+1}(p)$ (where as usual the index is to be read modulo $n$).

We have $H^1_1(p) = 0$ and $H^1_\beta(p)$ is a $q^2 \times q^2$ diagonal matrix that can be partitioned in blocks of size $q \times q$. The off-diagonal blocks are $0$ and we denote by $H^1_\beta(p)^k$ the $k$-th diagonal block that is itself a diagonal matrix. In particular, the $\ell$-th diagonal entry of $H^1_\beta(p)^k$ is \[\frac{\partial^2}{\partial (x_\beta)_k^\ell \partial (x_1)^k_\ell} IMM^n_q (p).\]

We obtain for $\beta \neq 1$
\[
 H^1_\beta(p)^k = \left\{ \begin{array}{cc}
                           \left[\begin{array}{cccc}
                           1& & & \\
                            & \ddots & & \\
                            & & 1 & \\
                            & & & \omega^{\beta-2}
                          \end{array} \right]& \text{if $k\neq q$}, \\
                           \left[\begin{array}{cccc}
                           \omega^{n-\beta}& & & \\
                           & \ddots & & \\
                           & & \omega^{n-\beta} & \\
                           & & & \omega^{n-2}
                          \end{array} \right]& \text{if $k=q$}.\end{array}
\right.
\]

\begin{prop}
 The matrix $H(p)$ is non-singular.
 \begin{proof}
  Let $a_n = \frac{(q-1)^{n-1} + (-1)^n}{q}$ and define a matrix $C$ having the same block structure as $H(p)$, with 
\[
 (C^1_1)^k = \left\{ \begin{array}{cc}
                           \left[\begin{array}{cccc}
                           -\frac{n-2}{n-1}& & & \\
                            & \ddots & & \\
                            & & -\frac{n-2}{n-1} & \\
                            & & & \frac{a_{n-1}}{a_n}\omega
                          \end{array} \right]& \text{if $k\neq q$}, \\
                           \left[\begin{array}{cccc}
                           \frac{a_{n-1}}{a_n}& & & \\
                           & \ddots & & \\
                           & & \frac{a_{n-1}}{a_n} & \\
                           & & & \frac{n-2}{(q-1)(n-1)}\omega^{2}
                          \end{array} \right]& \text{if $k=q$},\end{array}
\right.
\]  
and for $\beta \geq 2$
\[
 (C^1_\beta)^k = \left\{ \begin{array}{cc}
                           \left[\begin{array}{cccc}
                           \frac{1}{n-1}& & & \\
                            & \ddots & & \\
                            & & \frac{1}{n-1} & \\
                            & & & \frac{(-1)^n}{a_n}\omega^{(n-1)(n-\beta)}
                          \end{array} \right]& \text{if $k\neq q$}, \\
                           \left[\begin{array}{cccc}
                           \frac{(-1)^n}{a_n} \omega^{(n-1)(\beta-2)}& & & \\
                           & \ddots & & \\
                           & & \frac{(-1)^n}{a_n} \omega^{(n-1)(\beta-2)} & \\
                           & & & \frac{1}{(q-1)(n-1)}\omega^{2}
                          \end{array} \right]& \text{if $k=q$}.\end{array}
\right. 
\]
The matrix $C$ is the inverse of $H(p)$.

The proof of this fact reduces to the calculation of the product $H(p) \cdot C$. Let $R = H(p) \cdot C$ and let $R^\alpha _\beta$ be the $q^2 \times q^2$ blocks of $R$. For every $\alpha,\beta$, we have
\[
 R^\alpha_\beta = \textsum_\gamma H^\alpha_\gamma(p) \cdot C^\gamma_\beta = \textsum_\gamma H^1_{\gamma-\alpha+1}(p) \cdot C^1_{\beta - \gamma +1}.
\]
As every $q^2 \times q^2$ is diagonal, this calculation reduces to the $q\times q$ diagonal blocks $(R^\alpha_\beta)^k$ of $R^\alpha_\beta$:
\[
 ( R^\alpha_\beta)^k = \textsum_\gamma H^1_{\gamma-\alpha+1}(p)^k \cdot (C^1_{\beta - \gamma + 1})^k.
\]
A straightforward calculation shows that $(R^\alpha_\beta)^k$ is the $q\times q$ identity matrix if $\alpha = \beta$ and $0$ otherwise.
 \end{proof}
\end{prop}

By B. Segre dimension formula, we conclude:

\begin{thm}\label{thm: dual of calImm is hypersurface}
The dual variety $(\calI mm_q^n)^\vee$ of $\calI mm_q^n$ is a hypersurface in $\bbP V^*$.
\end{thm}

\section{The singular locus of $IMM^n_q$}

In this section we investigate geometric properties of the singular locus of the hypersurface $\calI mm^n_q$. We refer to \cite{EisHar:3264} for general results on the singular locus of general polynomials. For a vector space $W$, the hypersurface cut out by a general polynomial $g \in S^d W^*$ is a smooth variety in $\bbP W$; the set of polynomials whose zero sets are singular is a hypersurface in $\bbP S^d W^*$, called the discriminantal hypersurface. The general polynomial $g$ of the discriminantal hypersurface cuts out a variety, whose singular locus is an isolated double point (see \cite[Ch. 7]{EisHar:3264}).

In this section, $\calI mm^n_q$ denotes the affine cone in $V^*$ over the projective variety $V(IMM^n_q) \subseteq \bbP V^*$, and $\calS ing^n_q$ denotes the singular locus of $\calI mm^n_q$. 

\begin{obser}\label{obser: equations for sing}
 We have 
 \[
  \frac{\partial}{\partial (x_\alpha)^i_j} IMM^n_q = (X_{\alpha-1} \cdots X_1 X_n\cdots X_{\alpha+1})^j_i \in S^{n-1}V.
 \]
Therefore
\[
 \calS ing^n_q = \{(X_1 \vvirg X_n)\in \bbP V^* : X_{\alpha-1}\cdots X_{\alpha+1} = 0  \ \forall \alpha=1 \vvirg n\}.
\]
\end{obser}

We will prove that $\calS ing^n_q$ is a reducible variety whose irreducible components can be expressed as orbit-closures of certain quiver representations of the equioriented quiver $\tilde{\bfA}_n$, that is
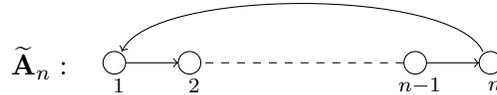
\[\begin{tikzpicture}
	\draw (-1,3) node{${\tilde{\bfA}_n:}$};

	\draw[->] (0, 3) -- (0.85,3);
	\draw[dashed] (1, 3) -- (4,3);
	\draw[->] (4, 3) -- (4.85,3);

	\draw[<-] (0.1,3.15) .. controls (0.5,4) and (4.5,4) .. (4.9,3.15);

	\draw[fill=white] (0,3) circle(.15);
	\draw[fill=white] (1,3) circle(.15);
	\draw[fill=white] (4,3) circle(.15);
	\draw[fill=white] (5,3) circle(.15);

	\draw[yshift=-.3cm] (0,3) node{$~_1$};
	\draw[yshift=-.3cm] (1,3) node{$~_2$};
	\draw[yshift=-.3cm] (4,3) node{$~_{n-1}$};
	\draw[yshift=-.3cm] (5,3) node{$~_{n}$};
\end{tikzpicture}
\]

Moreover, we determine desingularizations of these irreducible components, that will allow us to give a formula for the dimension of the singular locus.

We first recall some basic facts about representation theory of quivers; we refer to \cite{AsSiSk1} for a complete presentation of the subject.

A quiver is a finite directed graph; in our notation a quiver is $\calQ = (Q_0,Q_1)$ where $Q_0$ is the set of the vertices of the graph and $Q_1$ is the set of the arrows, given as ordered pairs of vertices; let $Q_r$ denote the set of paths of length $r$, namely the set of $r$-tuples of consecutive edges. A representation of $\calQ$ (or a $\calQ$-representation) is a pair $M = (\{M_a\}_{a \in Q_0}, \{ \phi_\alpha\}_{\alpha \in Q_1})$ where $M_a$ is a finite dimensional $\bbC$-vector space and, if $\alpha$ is an arrow from $a$ to $b$, then $\phi_\alpha \in \Hom(M_a,M_b)$; a map between two representations $M,N$ of a quiver $\calQ$ is a collection of maps $\{f_a : a \in Q_0\}$, with $f_a : M_a \to N_a$, that are compatible with the maps of $M$ and $N$.

Let $\frakMod_\calQ$ denote the category of finite dimensional representations of $\calQ$; the direct sum of two representations $M,N$ is defined naturally as $(M \oplus N)_a = M_a \oplus N_a$ with the maps defined component wise. A representation $E$ of $\calQ$ is called indecomposable if it cannot be written as a proper direct sum of two representations. $\frakMod_\calQ$ is a Krull-Schmidt category, namely every representation has a unique decomposition into indecomposable representations. The dimension vector of a representation $M$ is the vector $\mathbf{dim}(M) = (\dim M_a)_{a \in Q_0} \in \bbN^{\vert Q_0 \vert}$.

If $w = \alpha_1 \cdots \alpha_k$ is a path on $\calQ$ from $a$ to $b$ and $M$ is a representation of $\calQ$, $\phi_w = \phi_{\alpha_k} \circ \cdots \circ \phi_{\alpha_1}$ denotes the map $M_a \to M_b$ defined by the composition of the maps on $w$. 

For a quiver $\calQ$, let $\bbC \calQ$ be its path algebra. A two-sided ideal $I$ of $\bbC \calQ$ is said to be admissible if $Q_k \subseteq I$ for some $k$ and $I$ is contained in the ideal generated by $Q_2$. A representation $M = (\{M_a\},\{\phi_\alpha\})$ of $\calQ$ bounded by an ideal $I$ is an element of $\frakMod_\calQ$ such that, for every $\omega \in I$ (that is given by a linear combination of paths) the homomorphism $\phi_\omega$ is identically $0$. We denote by $\frakMod_\calQ^I$ the category of representations of $\calQ$ bounded by $I$.

The category $\frakMod^I_\calQ$ is equivalent to the category of finite dimensional representations of the algebra $\bbC \calQ /I$ (see \cite[Ch. III, Thm. 1.6]{AsSiSk1}, for a proof).

For a fixed dimension vector $\bfd = (d_a)_{a \in Q_0}$, a representation of $\calQ$ is given, up to isomorphism, by a collection of maps $\phi_\alpha : \bbC^{d_a} \to \bbC^{d_b}$, for any $\alpha$ arrow from $a$ to $b$. More precisely, if $\frakMod_\calQ(\bfd)$ denotes the set of representations with dimension vector $\bfd$, we have
\[
\frakMod_\calQ(\bfd) = \bigoplus_{\substack{\alpha \in Q_1 \\ \alpha: a\to b}} \Hom(\bbC^{d_a} , \bbC^{d_b});
\]
this is called the representation variety of $\calQ$. If $I$ is an admissible ideal of $\bbC \calQ$, then $\frakMod_\calQ^I(\bfd)$ is an algebraic subvariety of $\frakMod_\calQ(\bfd)$.

The group $G(\calQ,\bfd) := \prod_{a \in Q_0} GL(\bbC^{d_a})$ acts on $\frakMod_\calQ(\bfd)$ by change of basis. Two representations $M,N \in \frakMod_\calQ(\bfd)$ are in the same orbit under the action of $G(\calQ)$ if and only if they have the same decomposition into indecomposables if and only if they are isomorphic as $\calQ$-representations. Given $M \in \frakMod_\calQ(\bfd)$, we denote by $\scrO_M$ its orbit under the action of $G(\calQ,\bfd)$. Given two elements $M,N \in \frakMod_\calQ(\bfd)$, we say that $N$ is a degeneration of $M$ if $N \in \bar{\scrO}_M$ (where the closure is equivalently Zariski or euclidean): in this case we write $N \leq_{deg} M$. The relation $\leq_{deg}$ defines an order relation on the set of isomorphism classes of $\calQ$-representations with representation vector $\bfd$, or equivalently on the set of the orbits of the action of $G(\calQ,\bfd)$. Notice that if $N \leq_{deg} M$, then $\bar{\scrO}_N$ is a closed subvariety of $\bar{\scrO}_M$.

A different order relation on the set of $\calQ$-representations can be given as follows: if $M = (\{M_a\},\{\phi_\alpha\})$ and $N =(\{N_a\},\{\psi_\alpha\})$, we say that $N \leq_{rk} M$ if, for every path $w$ in $\calQ$, $\rk \phi_w \leq \rk \psi_w$. Prop. 2.1 in \cite{Ried:deg_quivers_with_relations} proves that $N \leq_{deg} M$ implies $N \leq_{rk} M$. Moreover, Corollary of Theorem 1 in \cite{Zwa:Deg_mods_rep_fin_algebras} proves that if $\bbC \calQ / I$ is of finite representation type (namely it has only finitely many indecomposable representations, up to isomorphism) then $\leq_{deg}$ and $\leq_{rk}$ are equivalent.

We say that $M$ is $k$-nilpotent if, for every path $w$ of length greater or equal to $k$, we have $\phi_w = 0$. Equivalently a $k$-nilpotent representation is a representation bounded by the ideal generated by $Q_k$. We say that a representation $M$ is nilpotent if it is $k$-nilpotent for some $k$.

The set of $k$-nilpotent representations with a fixed dimension vector $\bfd$ is an algebraic variety in $\frakMod_\calQ(\bfd)$ that we denote by $\frakMod^{[k]}_\calQ(\bfd)$. The union of all such varieties is denoted by $\frakMod^{[nil]}_\calQ(\bfd)$ and it is the set of all nilpotent representations of dimension vector $\bfd$. Similarly, we will use the notation $\frakMod^{[k]}_\calQ$ for the category of $k$-nilpotent representations of $\calQ$ and $\frakMod^{[nil]}_\calQ$ for the category of all nilpotent representations.

Write $(q^n) := (\underbrace{q\vvirg q}_{n \text{ times}})$. Then $\frakMod_{\tilde{\bfA}_n}((q^n)) \simeq V^* $, directly from the definition of representation variety. In order to determine the irreducible components of $\calS ing^n_q$, we will study the orbits and the orbit closures of the action of $G(\tilde{\bfA}_n,(q^n))$ on $\frakMod^{[n-1]}_{\tilde{\bfA}_n}(q^n)$. First, we report the classification of the indecomposable $(n-1)$-nilpotent representations of $\tilde{\bfA}_n$:

\begin{thm}[see e.g. \cite{Schif:typeAflags}]\label{thm: classification of n-1 nilp indecomp}
 The $(n-1)$-nilpotent indecomposable representations for the quiver $\tilde{\bfA}_n$ are in one-to-one correspondence with pairs $(\alpha,\beta)$ with $\alpha,\beta \in \bbZ_n$, $\alpha \neq \beta +1$. We denote by $E_{\alpha\beta}$ the indecomposable representation corresponding to the pair $(\alpha,\beta)$. We have
 \[
(E_{\alpha\beta})_\gamma = \left\{\begin{array}{ll} 
                            \bbC &  \text{if $\gamma \in [\alpha,\beta]$},\\
                            0 & \text{otherwise,}  
                           \end{array}\right.
\]
 with the maps given by $id_\bbC : (E_{\alpha\beta})_\gamma \to (E_{\alpha\beta})_{\gamma+1}$ if $\gamma,\gamma+1 \in [\alpha,\beta]$ and $0$ otherwise.
 
 In particular $\bbC \tilde{\bfA}_n / \langle Q_{n-1}\rangle$ is of finite-representation type.
\end{thm}

Following \cite{AbDf:Deg_equioriented_Am} and \cite{AbDf:Deg_Am}, define another order relation on $\frakMod^{[n-1]}_{\tilde{\bfA}_n}(q^n)$: if $\alpha_1 \vvirg \alpha_k$ are (indices of) vertices of $\tilde{\bfA}_n$, we write $[\alpha_1 \vvirg \alpha_m]$ to indicate that the vertices lie in this order on $\tilde{\bfA}_n$ (but they are not necessarily consecutive) traveling on a single cycle on $\tilde{\bfA}_n$ starting from $\alpha_1$; we will explicitly state if some of them are allowed to be equal. We say that $M \leq_{ad} N$ if there is a finite sequence $M = M_0 \vvirg M_u = N \in \frakMod^{[n-1]}_{\tilde{\bfA}_n}(q^n)$ such that $M_{j+1}$ can be obtained from $M_j$ by one of the following elementary operations:

\begin{itemize}
 \item[$\cdot$] \emph{gluing}: for vertices $[\alpha, \beta, \gamma]$ (with $\beta$ possibly equal to $\alpha$), substitute $E_{\alpha,\beta} \oplus E_{\beta+1,\gamma}$ with $E_{\alpha,\gamma}$;
 \item[$\cdot$] \emph{shift}: for vertices $[\alpha, \gamma , \beta , \delta]$, substitute $E_{\alpha,\beta} \oplus E_{\gamma,\delta}$ with $E_{\alpha,\delta} \oplus E_{\gamma,\beta}$.
\end{itemize}

The same argument used in Theorem 3.2 of \cite{AbDf:Deg_equioriented_Am} shows that the order $\leq_{ad}$ is equivalent to the orders $\leq_{rk}$ and $\leq_{deg}$. Therefore, Observation \ref{obser: equations for sing}, the results of \cite{Ried:deg_quivers_with_relations} and \cite{Zwa:Deg_mods_rep_fin_algebras} and Theorem \ref{thm: classification of n-1 nilp indecomp} provide the following result:

\begin{thm}\label{thm: sing compt iff rk-maximal}
 The singular locus $\calS ing^n_q$ is $\frakMod^{[n-1]}_{\tilde{\bfA}_n}(q^n)$, regarded as an algebraic subvariety of $\frakMod_{\tilde{\bfA}_n}(q^n) \simeq V^*$. 
 
  The irreducible components of $\calS ing ^n_q$ are in one-to-one correspondence with the set of $\leq_{rk}$-maximal (and equivalently $\leq_{ad}$-maximal, and equivalently $\leq_{deg}$-maximal) elements in the set of $(n-1)$-nilpotent representations of $\tilde{\bfA}_n$, up to isomorphism.
  
 More precisely, if $M$ is a maximal element, then $\bar{\scrO}_M$ is an irreducible component of $\calS ing ^n_q$, all irreducible components arise in this way and if $M,N$ are not isomorphic, then $\bar{\scrO}_M \neq \bar{\scrO}_N$.
\end{thm}

Our next goal is to determine a formula for the dimension of each irreducible component of $\calS ing^n_q$.

For $M \in \frakMod^{[n-1]}_{\tilde{\bfA}_n}(q^n)$, define an integer matrix $\mathbf{Rk}(M)$ encoding the ranks of the maps in $M$ on the paths of $\tilde{\bfA}_n$. If $M = (\{M_\alpha\},\{\phi_\alpha\})$, define $(\mathbf{Rk}(M))_{\alpha\beta} = r_{\alpha\beta}$ where $r_{\alpha\beta} := \rk ( \phi_{\beta-1} \circ \cdots \circ \phi_\alpha : U_\alpha \to U_\beta )$; clearly $r_{\alpha+1,\alpha} = 0$ and we set $r_{\alpha,\alpha} = q$. The combinatorial relations to which the $r_{\alpha\beta}$'s are subject are not trivial.

 Let $\calZ^n_q(M) := \bar{\scrO}_M$ be the irreducible component of $\calS ing^n_q$ corresponding to the $\leq_{rk}$-maximal element $M$. We will give a formula for $\dim \calZ^n_q(M)$ in terms of the entries of $\mathbf{Rk}(M)$. The strategy to compute such formula is standard: we realize a desingularization of $\calZ^n_q(M)$ as a vector bundle $E \to F$ over a product of flag varieties so that we have $\dim \calZ^n_q(M) = \dim E = \dim F + \dim E_p$, where $E_p$ is the fiber of the bundle $E$ over a point $p \in F$.

For non-negative integers $m$ and $0 = k_0 \leq \cdots \leq k_{\ell+1} = m$, denote by $Flag_{k_0 \vvirg k_{\ell+1}}$ the flag variety of nested sequences of subspaces of dimension $k_1 \vvirg k_\ell$ in $\bbC^m$, namely
\[
 Flag_{k_0 \vvirg k_{\ell+1}} := \{ \bfF = (F_0 \vvirg F_{\ell+1}) : 0 = F_0 \subseteq F_1 \subseteq \cdots \subseteq  F_{\ell+1} = \bbC^m, \dim F_j = k_j \}.
\]
The dimension of a flag variety is given by the following formula (see e.g. \cite{Bri:Flags}):
\[
 \dim Flag_{k_0 \vvirg k_{\ell+1}} =  \sum_{j=1}^{\ell+1} (k_j - k_{j-1})k_{j-1}.
\]

Given a $\leq_{rk}$-maximal representation $M \in \frakMod^{[n-1]}_{\tilde{\bfA}_n}(q^n)$ and its corresponding matrix of ranks $\mathbf{Rk}(M)$, define
 \begin{align*}
 \calF_\alpha &:= Flag_{r_{\alpha+1,\alpha},r_{\alpha+2,\alpha} \vvirg r_{\alpha-1,\alpha},r_{\alpha,\alpha}} = \\
 &= \{\bfF_\alpha = (W^\alpha_{1} \vvirg W^\alpha_n) : W^\alpha_\gamma \subseteq W^\alpha_{\gamma +1}, \dim W^\alpha_\gamma = r_{\alpha+\gamma,\alpha}\}.
\end{align*}
In particular, for any $\bfF_\alpha \in \calF_\alpha$, we have $W^\alpha_1 = 0$ and $W^\alpha_n = \bbC^q$. Let $\calF_M := \calF_1 \ttimes \calF_n$ and define
\[
 \tilde{\calZ^n_q(M)} := \{((\bfF_\alpha)_\alpha , X) \in \calF_M \times V^* : X_\alpha W^\alpha_\gamma \subseteq W^{\alpha + 1}_{\gamma -1}, \ \forall \alpha =1\vvirg n, \forall \gamma =2 \vvirg n\}.
\]

\begin{thm}\label{thm: desingularization of irred compts of Sing}
 $ \tilde{\calZ^n_q(M)}$ is a vector bundle over $\calF_M$ and its projection to $V^* \simeq \frakMod_{\tilde{\bfA}_n}(q^n)$ defines a desingularization of $\calZ^n_q(M)$.
Moreover the dimension of the irreducible component $\calZ^n_q(M)$ is 
\[
\dim \calZ^n_q(M) = \sum_{\alpha = 1}^n \sum_{\beta=2}^n (r_{\alpha + \beta, \alpha} - r_{\alpha+\beta-1,\alpha})(r_{\alpha+\beta-1,\alpha} + r_{\alpha+\beta,\alpha+1}).
\]

\end{thm}
\begin{proof}
The proof of the first part of the theorem is essentially the same as Lemma 6.2 in \cite{AbDfKr:Geo_Am}.

The second part of the Theorem is obtained by computing the dimension of the vector bundle $\pi: \tilde{\calZ^n_q(M)} \to \calF_M$ as
\[
 \dim \tilde{\calZ^n_q(M)} = \dim \calF_M + \dim \pi^{-1}(p).
\]
\end{proof}

\textsc{Examples}

We conclude this section showing some explicit calculation of the decomposition of $\calS ing^n_q$ into its irreducible components for small values of $n,q$. We consider the cases $n\leq 3$ (for any $q$) and the case $q=2$ (for any $n$).

For these cases, we determine the $\leq_{deg}$-maximal $(n-1)$-nilpotent representations of $\tilde{\bfA}_n$ of dimension vector $(q^n)$; moreover, for each maximal $M$ we determine an element of $V^* \simeq \frakMod_{\tilde{\bfA}_n}(q^n) \simeq (Mat_q)^{\oplus n}$ whose orbit-closure under the action of $G(\tilde{\bfA}_n,(q^n)) \simeq GL_q^{\times n}$ is the corresponding irreducible component $\calZ^n_q(M)$ of $\calS ing^n_q$ as well as the dimension of $\calZ^q_n(M)$.

We will use the notation $E_{\alpha\beta}$ as in Theorem \ref{thm: classification of n-1 nilp indecomp} for the $(n-1)$-nilpotent indecomposable representations of $\tilde{\bfA}_n$.

\subsection*{Cases $\mathbf{n=1,2}$} These cases are trivial. 

If $n=1$, then $IMM^n_q$ is linear, and therefore $\calI mm ^1_q$ is just the hyperplane of traceless elements in $V^*\simeq U_1^* \otimes U_1$, that is smooth.

If $n=2$, then $IMM^2_q$ is a quadric of maximal rank, and therefore its only singular point is the origin of $V^*$. Therefore $\calS ing^2_q$ is the origin of $V^*$ for every $q$. This corresponds to the representation $\mathbf{0}$ of $\tilde{\bfA}_2$ having dimension vector $(q,q)$ and decomposition into indecomposables $\mathbf{0} = E_{11}^{\oplus q} \oplus E_{22}^{\oplus q}$. The corresponding matrix of ranks is the $2 \times 2$ identity matrix, and the formula for the dimension in Theorem \ref{thm: desingularization of irred compts of Sing} gives $0$.

\subsection*{Case $\mathbf{n=3}$} There are six $2$-nilpotent indecomposable representations for $\tilde{\bfA}_3$. With the notation of Theorem \ref{thm: classification of n-1 nilp indecomp} they are: $E_{11},E_{22},E_{33},E_{12},E_{23},E_{31}$. Notice that no shifting operations may be performed among them. Moreover, no gluing operation may be performed between any of $E_{12},E_{23},E_{31}$ and any other indecomposable: in particular if $N$ is a $\leq_{deg}$-maximal $2$-nilpotent representation of $\tilde{\bfA}_3$ with dimension vector $(d_1,d_2,d_3)$, then $N \oplus E_{12}$ is maximal with dimension vector $(d_1+1,d_2+1,d_3)$ and similarly for $N \oplus E_{23}$ and $N \oplus E_{31}$.

We claim that, if $M$ is a $\leq_{deg}$-maximal $2$-nilpotent representation of $\tilde{\bfA}_3$ with dimension vector $(q,q,q)$, then
\[
 M = (E_{\alpha\alpha} \oplus E_{\alpha +1,\alpha-1})^{\oplus k} \oplus (E_{12} \oplus E_{23} \oplus E_{31})^{\oplus \ell}
\]
for some $\alpha=1,2,3$ and some $k,\ell$ such that $q = k +2\ell$. It is clear that if $M$ has this form then it is $\leq_{deg}$-maximal. To show that all maximals have this form, it suffices to observe the following two facts:

\begin{itemize}
\item[$\cdot$] if two distinct indecomposables among $E_{11},E_{22},E_{33}$ occur in $N$ then $N$ is not maximal because a gluing operation among those two indecomposables may be performed;

\item[$\cdot$] if $E_{12} \oplus E_{23}$ occurs in a maximal representation $N$ of dimension vector $(q,q,q)$, then $E_{31}$ occurs as well; indeed the dimension vectors obtained from indecomposables of the form $E_{12}, E_{23}$ is of the form $(d,d+d',d')$; if $E_{31}$ does not occur, then both $E_{11}$ and $E_{33}$ occur, in contradiction with the previous observation. The analogous statements hold for $E_{23} \oplus E_{31}$ and $E_{31} \oplus E_{12}$.
\end{itemize}

If $M$ is maximal of dimension vector $(q,q,q)$, write $M = M' \oplus (E_{12} \oplus E_{23} \oplus E_{31})^{\oplus \ell}$, with no summands of the form $ (E_{12} \oplus E_{23} \oplus E_{31})$ occurring in $M'$. Then only one among $E_{12},E_{23}$ or $E_{31}$ appears in $M'$, so $M'= (E_{\alpha\alpha} \oplus E_{\alpha +1,\alpha-1})^{\oplus k}$ for some $\alpha$.

An easy combinatorial argument shows that the number of irreducible components is $3 (q-1)/2$ if $q$ is odd and $1 + 3q/2$ if $q$ is even.

We compute the matrices of ranks of maximal representations of dimension vector $(q,q,q)$. For simplicity, fix $\alpha = 3$; the matrix of ranks of $M = (E_{33} \oplus E_{12})^{\oplus k} \oplus (E_{12} \oplus E_{23} \oplus E_{31})^{\oplus \ell}$ is
\[
 \bfR\bfk (M) = \left( \begin{array}{ccc}
                        q & \ell & 0 \\ 0 & q & \ell \\ \ell +k & 0 & q
                       \end{array} \right).
\]

The irreducible component $\calZ^3_q(M)$ is the orbit closure in $V^* \simeq (Mat_q)^{\oplus 3}$ under the action of $G(\tilde{\bfA}_3,(q,q,q)) = (GL_q)^{\times 3}$ of the point 
\[
 \left( 
\left( \begin{array}{ccc}
         0 & I_\ell & 0 \\
         0 & 0 & 0 \\
         0 & 0 & 0 \\
        \end{array} \right) , \left(  \begin{array}{ccc}
         0 & I_\ell & 0 \\
         0 & 0 & 0 \\
         0 & 0 & 0 \\         
        \end{array}\right) , \left(  \begin{array}{ccc}
         0  & I_\ell & 0 \\
         0 & 0 & 0 \\
         0 & 0 & I_k \\         
        \end{array}\right) \right) \in V^*,
\]
where the blocking is $(\ell,\ell,k) \times (\ell,\ell,k)$.

Theorem \ref{thm: desingularization of irred compts of Sing} provides $\dim \calZ^3_q(M) = q^2 + 2\ell (q-\ell)$. In particular, if $q$ is even, consider $M = (E_{12} \oplus E_{23} \oplus E_{31})^{\oplus q/2}$ so that $\dim \calS ing^3_q = \dim \calZ^3_q(M) = \frac{3}{2}q^2$. If $q$ is odd, consider $M = (E_{33} \oplus E_{12}) \oplus (E_{12} \oplus E_{23} \oplus E_{31})^{\oplus (q-1)/2}$, so that $\dim \calS ing^3_q = \dim \calZ^3_q(M) = (3q^2 -1)/2 $.

\subsection*{Case $\mathbf{q=2}$} Suppose $M$ is a maximal $\tilde{\bfA}_n$ representation of dimension vector $(2^n)$, with $n \geq 4$. It is easy to observe that at most $4$ indecomposable summands occur in $M$. Notice that $V^* \simeq (Mat_2)^{\oplus n}$ and $G(\tilde{\bfA}_n,(2^n)) \simeq GL_2^{\times n}$.

If exactly two distinct indecomposable summands occur in $M$, then
 \[
  M \simeq (E_{\alpha+1,\beta} \oplus E_{\beta+1,\alpha})^{\oplus 2}
 \]
 for distinct vertices $\alpha,\beta$. There are $\binom{n}{2}$ such maximals, because any choice of two vertices provides one of them; the irreducible component $\calZ^n_2(M)$ is the orbit-closure of $(X_1 \vvirg X_n) \in V^*$ where 
 \begin{align*}
X_\alpha &= X_\beta = 0,\\ 
X_\eps &= I_2 \quad \text{ if } \eps \neq \alpha,\beta. 
 \end{align*}
The codimension of $\calZ^n_2(M)$ in $V^*$ is $8$.

If exactly three distinct indecomposable summands occur in $M$, then
 \[
  M \simeq E_{\alpha+1,\gamma} \oplus E_{\beta+1,\alpha} \oplus E_{\gamma+1,\beta}
 \]
for distinct vertices $[\alpha,\beta,\gamma]$. There are $\binom{n}{3}$ such maximals because every choice of three vertices provides one of them. The irreducible component $\calZ^n_2(M)$ is the orbit-closure of $(X_1 \vvirg X_n) \in V^*$ where
\begin{align*}
X_\alpha &= X_\beta = X_\gamma = \left(\begin{smallmatrix}
                                              0 & 1 \\ 0 & 0
                                             \end{smallmatrix} \right),\\
X_\eps &= I_2 \quad \text{ if } \eps \neq \alpha,\beta,\gamma. 
\end{align*}

The calculation of the matrix of rank and Theorem \ref{thm: desingularization of irred compts of Sing} provide that $\calZ^n_2(M)$ has codimension $6$ for every $n$.

If exactly four distinct indecomposable summands occur in $M$, then
\[
 M \simeq E_{\alpha+1,\beta} \oplus E_{\beta+1,\alpha} \oplus E_{\gamma+1,\delta} \oplus E_{\delta+1,\gamma}
\]
for distinct vertices $[\alpha,\beta,\gamma,\delta]$. There are $2\binom{n}{4}$ such maximals because every choice of four vertices provides two of them. The irreducible component $\calZ^n_2(M)$ is the orbit closure of $(X_1 \vvirg X_n) \in V^*$ where
\begin{align*}
X_\alpha &= X_\beta =  \left(\begin{smallmatrix}                    1 & 0 \\ 0 & 0
                                             \end{smallmatrix}\right),\\
 X_{\gamma} &= X_\delta = \left(\begin{smallmatrix}
                                              0 & 0 \\ 0 & 1
                                             \end{smallmatrix}\right),\\
    X_\eps &= I_2 \quad \text{if } \eps \neq \alpha,\beta,\gamma,\delta. 
\end{align*}
The calculation of the matrix of ranks and Theorem \ref{thm: desingularization of irred compts of Sing} provide that $\calZ^n_2(M)$ has codimension $6$ for every $n$.
                                            
\section{The $(n-2)$-nd Jacobian locus of $\calI mm^n_q$}

In this section we study the $(n-2)$-nd Jacobian locus of $\calI mm_q^n$, that we denote by $\calW$, namely the zero locus of the partials of order $n-2$ of $IMM_q^n$. The goals are similar to the ones of the previous section: we want to determine the decomposition of $\calW$ into irreducible components and the dimensions of the components. Although it might be possible to study this variety in terms of quiver representations, we find a direct approach easier and more natural. 

Again, $\calI mm^n_q$ here denotes the affine cone in $V^*$ over the projective variety $V(IMM_q^n) \subseteq \bbP V^*$. By definition $\calW = V(J)$, where $J$ is the ideal generated by all partials of order $n-2$ of $IMM_q^n$:
\[
 J = \langle \Uptheta \cdot IMM_q^n: \Uptheta \in S^{n-2} V^*\rangle.
\]

Let $\Uptheta $ be a monomial of degree $(n-2)$ in the linear differential operators $(\xi_\alpha)^i_j = \frac{\partial}{\partial (x_\alpha)^j_i}$, that is a basis of $V^*$ dual to the basis of $V$ given by the $(x_\alpha)^i_j$'s. It is immediate that $\Uptheta \cdot IMM^n_q = 0$, unless $\Uptheta$ has one of the following two forms:
\begin{itemize}
\item[$\cdot$] $\Uptheta$ involves exactly one variable from each $A_\gamma$, except for two consecutive ones, $A_\alpha,A_{\alpha+1}$, and it has the form
\[
 \Uptheta =  (\xi_{n})_{i_1}^{i_n}\cdot (\xi_{n-1})^{i_{n-1}}_{i_n} \cdots (\xi_{\alpha+2})_{i_{\alpha+3}}^{t} \cdot (\xi_{\alpha-1})_s^{i_{\alpha-1}} \cdots (\xi_2)^{i_2}_{i_3}\cdot (\xi_1)^{i_1}_{i_2}
\]
for some $\alpha$ and some $i_\gamma,s,t$: we obtain
\[
\Uptheta \cdot IMM_q^n = \textsum_{i_{\alpha+1}} (x_{\alpha+1})^t_{i_{\alpha+1}} (x_{\alpha})^{i_{\alpha+1}}_s = (X_{\alpha +1} X_\alpha )_s^t;
\]

\item[$\cdot$] $\Uptheta$ involves exactly one variable from each $A_\gamma$, except for two non-consecutive ones, $A_\alpha,A_\beta$, and it has the form
 \[
 \Uptheta =  (\xi_n)_{i_1}^{i_n}\cdot (\xi_{n-1})^{i_{n-1}}_{i_n} \cdots  (\xi_{\beta+1})_{i_{\beta +2}}^{t} \cdot (\xi_{\beta-1})_{s}^{i_{\beta -1}} \cdots (\xi_{\alpha+1})_{i_{\alpha +2}}^{r} \cdot (\xi_{\alpha-1})_{p}^{i_{\alpha -1}} \cdots (\xi_2)^{i_2}_{i_3}\cdot(\xi_1)_{i_2}^{i_1}
 \]
for some $\alpha,\beta$ with $\alpha \neq \beta, \alpha \neq \beta \pm 1$ and some $i_\gamma,p,r,s,t$: we obtain
\[
\Uptheta \cdot IMM_q^n =  (x_\beta)^t_s (x_\alpha)^r_p.
\]
\end{itemize}

Hence, the $(n-2)$-nd Jacobian locus of $\calI mm_q^n$ is the variety in $V^*$ that is cut out by the following two sets of equations:
 \begin{equation*}
\left\{ \begin{array}{l}
(x_\beta)^t_s (x_\alpha)^r_p \\  \alpha ,\beta = 1\vvirg n; \alpha \neq \beta \pm 1\\
  p,r,s,t = 1\vvirg q
 \end{array}\right.
\qquad
\left\{\begin{array}{ll}
 (X_{\alpha + 1}X_\alpha )_s^t = 0\\
 \alpha = 1\vvirg n\\
 s,t = 1 \vvirg q.
\end{array}\right.
\end{equation*}

The first set of equations implies that if $(X_1 \vvirg X_n) \in \calW$, then only two consecutive $X_\alpha$'s may be non-zero; the second set of equations implies that, whenever two consecutive $X_{\alpha}$'s are non-zero, then their product must be $0$. Denote by $L_\alpha$ the linear space given by $L_\alpha := \{ (X_1 \vvirg X_n) : X_\beta= 0 \text{ if } \beta \neq \alpha,\alpha-1\} \subseteq V^*$ and $\calW_\alpha := \calW \cap L_\alpha$. We have
\[
 \calW = \bigcup_{\alpha=1 \vvirg n} \calW_\alpha.
\]
  
Define $\calW_{\alpha,r} := \{ (X_1 \vvirg X_n) \in \calW_\alpha: \rk X_{\alpha-1} \leq r, \rk X_\alpha \leq q-r\}$.  

Equivalently
\[
 \calW_{\alpha,r} = \{ (0\vvirg 0,X_{\alpha -1},X_\alpha,0\vvirg 0): \rk X_{\alpha -1} \leq r,\rk X_{\alpha} \leq q-r, X_{\alpha}X_{\alpha-1} = 0\}.
\]
It is clear that $\calW_{\alpha} = \bigcup_{r=0}^{q} \calW_{\alpha,r}$ and $\calW_{\alpha-1,0} = \calW_{\alpha,q}$. 

The following result completely describes the irreducible components of $\calW$ and their dimensions
\begin{thm}\label{thm: n-2 singular locus}
For every $\alpha =1 \vvirg n$, $r =0 \vvirg q$, $\calW_{\alpha,r}$ is an irreducible variety of dimension
\[
 \dim \calW_{\alpha,r} = q^2 + rq - r^2.
\]
In particular $\calW$ has $nq$ irreducible components. 

The dimension of $\calW$ is given by the maximal dimension of the $\calW_{\alpha,r}$'s; it is attained at $r = \lfloor \frac{q}{2} \rfloor$ or equivalently $r = \lceil \frac{q}{2} \rceil$:
\[
\dim \calW =  \dim \calW_{\alpha,r} = \left\{ \begin{array}{ll}
                                  \frac{5}{4}q^2 & \text{if $q$ is even},\\ \\
                                  \frac{5}{4}q^2 - \frac{1}{4} & \text{if $q$ is odd}.
                                 \end{array}
\right.
\]
\end{thm}
\begin{proof}
For every $r$, $\calW_{\alpha,r}$ is isomorphic to the variety
\[
 Z := \{(X,Y) \in Mat_q \times Mat_q : \rk X \leq r, \rk Y \leq q-r, YX = 0\}.
\]
We will define a desingularization of $Z$ as a vector bundle over $Gr(r,\bbC^q)$, the Grassmannian of $r$-planes in $\bbC^q$ and we will compute the dimension of such desingularization.

Consider
\[
 \tilde{Z} := \{(X,Y,E) \in Mat_q \times Mat_q \times Gr(r,\bbC^q): \Im X \subseteq E \subseteq \ker Y \}.
\]
The projection $(X,Y,E) \mapsto E$ defines a vector bundle over $Gr(r,\bbC^q)$, therefore $\tilde{Z}$ is a smooth irreducible variety. The fiber of this bundle has dimension $q^2$.

On the other hand, the projection $(X,Y,E) \mapsto (X,Y)$ maps $\tilde{Z}$ surjectively onto $Z$ and the projection is one-to-one on the open set $\{(X,Y) : \rk X = r, \rk Y = q-r\}$; therefore $\tilde{Z}$ is a desingularization of $Z$. This proves that $Z$ is irreducible and so are the $\calW_{\alpha,r}$'s.

The dimension of $Z$ is
\[
 \dim Z = \dim \tilde{Z} = \dim \tilde{Z}_{E} + \dim Gr(r,\bbC^q) = q^2 + r(q - r),
\]
where $Z_E$ is the fiber of the bundle $\tilde{Z} \to Gr(r,\bbC^q)$ at a point $E \in Gr(r,\bbC^q)$.

It is clear that all the $\calW_{\alpha,r}$ are distinct, with the only exception of $\calW_{\alpha-1,0} = \calW_{\alpha,q}$, therefore $\calW$ has $nq$ irreducible components. The straightforward calculation of the maximal value of $q^2 + rq - r^2$ for an integer $r$ will conclude the proof.
\end{proof}
 
 \subsection*{Acknowledgements} I thank J.M. Landsberg for his constant guidance while I was writing this paper. I also thank C. Ikenmeyer, R. Kinser, S. Naldi, M. Phillipson, O. Schiffmann and J. Weyman for fruitful discussions, suggestions and comments. A significant part of this work was done while I was visiting the Simons Institute for the Theory of Computing at Berkeley during a semester long program on \emph{Algorithms and Complexity in Algebraic Geometry} in Fall 2014.

\bibliographystyle{amsalpha}
\bibliography{mybib}

\end{document}